\documentclass{article}

\usepackage{arxiv}

\usepackage[utf8]{inputenc} % allow utf-8 input
\usepackage[T1]{fontenc}    % use 8-bit T1 fonts
\usepackage{hyperref}       % hyperlinks
\usepackage{url}            % simple URL typesetting
\usepackage{booktabs}       % professional-quality tables
\usepackage{amsfonts}       % blackboard math symbols
\usepackage{nicefrac}       % compact symbols for 1/2, etc.
\usepackage{microtype}      % microtypography
\usepackage{lipsum}

\usepackage{morefloats}
\usepackage{float}
\usepackage{natbib}

\usepackage{color}
\usepackage{multirow}
\usepackage{latexsym}
\usepackage{amsmath,amssymb,amsthm,graphicx}
\usepackage{dsfont}
\usepackage{enumitem}
\usepackage{eufrak}
\usepackage{arydshln}

\newtheoremstyle{thm}% name
{9pt}%      Space above, empty = `usual value'
{9pt}%      Space below
{\itshape}% Body font
{}%         Indent amount (empty = no indent, \parindent = para indent)
{\bfseries}% Thm head font
{.}%        Punctuation after thm head
{ }% Space after thm head: \newline = linebreak
{}%         Thm head spec
\theoremstyle{thm}
\DeclareMathAlphabet{\mathscr}{OT1}{pzc}{m}{it}
\newtheorem{theorem}{Theorem}[section]
\newtheorem{lemma}[theorem]{Lemma}
\newtheorem{corollary}[theorem]{Corollary}

\newtheoremstyle{def}% name
{9pt}%      Space above, empty = `usual value'
{9pt}%      Space below
{}% Body font
{}%         Indent amount (empty = no indent, \parindent = para indent)
{\bfseries}% Thm head font
{.}%        Punctuation after thm head
{ }% Space after thm head: \newline = linebreak
{}%         Thm head spec
\theoremstyle{def}

\newtheorem{remark}[theorem]{Remark}

\newcommand{\R}{\mathbb{R}} % reelle
\newcommand{\N}{\mathbb{N}} % natuerliche
\newcommand{\E}{\mathbb{E}} %Erwartungswert
\newcommand{\PP}{\mathbb{P}} %Wahrscheinlichkeit
\renewcommand{\footnoterule}{%
	\kern -3.5pt
	\hrule width \textwidth height 1pt
	\kern 3.5pt
}

\makeatletter
\def\blfootnote{\xdef\@thefnmark{}\@footnotetext}
\makeatother

\title{A goodness-of-fit test for the Zeta distribution with unknown parameter}

\author{Bruno Ebner\\
Institute of Stochastics, \\
Karlsruhe Institute of Technology (KIT), \\
Englerstr. 2, 76133 Karlsruhe, Germany. \\
\href{mailto:Bruno.Ebner@kit.edu}{Bruno.Ebner@kit.edu}\\
\And
Daniel Hlubinka\\
Department of Probability and Mathematical Statistics,\\
Faculty of Mathematics and Physics, Charles University,\\
Sokolovská 83, 18675 Praha, Czech Republic. \\
\href{mailto:daniel.hlubinka@matfyz.cuni.cz}{daniel.hlubinka@matfyz.cuni.cz}
}

\begin{document}

\date{\today}
\maketitle

\blfootnote{ {\em MSC 2010 subject
classifications.} Primary 62G10 Secondary 62E10}
\blfootnote{{\em Key words and phrases} Goodness-of-fit; Zeta distribution; discrete data; Hilbert-space valued random elements; parametric bootstrap}

\begin{abstract}
We introduce a new goodness-of-fit test for count data on $\N$ for the Zeta distribution with unknown parameter. The test is built on a Stein-type characterization that uses, as Stein operator, the infinitesimal generator of a birth–death process whose stationary distribution is Zeta. The resulting
$L^2$-type statistic is shown to be omnibus consistent, and we establish the limit null behavior as well as the validity of the associated parametric bootstrap procedure. In a Monte Carlo simulation study, we compare the proposed test with the only existing Zeta-specific procedure of \citet{M:2009}, as well as with more general competitors based on empirical distribution functions, kernel Stein discrepancies and other Stein-type characterizations.    \end{abstract}

\section{Introduction}\label{sec:Intro}
In this paper we study goodness-of-fit testing for the Zeta (also called Riemann–Zeta, Zipf or discrete Pareto)
family of distributions. This one-parameter family emerges as a canonical model for
heavy-tailed count data on $\mathbb{N}$, with probability mass function
\[
  p_s(k) = \frac{k^{-s}}{\zeta(s)}, \qquad k \in \mathbb{N}, \; s>1,
\]
where $\zeta(s) = \sum_{k=1}^\infty k^{-s}$, $s>1$, denotes the Riemann zeta function. Zeta and related Zipf-type laws are widely
used in applications where empirical frequencies decay approximately like a power of the
rank, such as word counts in natural language, city-size distributions, and degree
distributions in large networks; see, for example,
\citet{Pareto1896,Lotka1926,Zipf1949,Mandelbrot1961,Gabaix1999,BarabasiAlbert1999,ClausetShaliziNewman2009,Newman2005}.
Despite this broad range of applications, methodological work on formal goodness-of-fit
tests specifically tailored to the Zeta family is rather limited. General omnibus procedures
based on empirical distribution functions or tail fitting can be adapted to the Zeta model,
but only few tests exploit its particular structure, and the behavior of different
procedures under realistic alternatives has, to date, not been systematically investigated.
The aim of this paper is to fill this gap by developing a new procedure including it's theory, and by comparing several Stein-based, Mellin-transform-based, and characterization-based tests for the Zeta law.

To be specific, we denote the Zeta family of distributions by $\mathcal{Z}=\{ \mathrm{Zeta}(s): s>1\}$, where $s$ is a shape parameter. Note that the moments of the Zeta law only exist for $s>2$ (mean) and $s>3$ (variance), so we potentially have a heavy tailed discrete distribution. This fact is one reason inference must be tail-aware \citep{JohnsonKotzKemp2005,Newman2005}. Let $X_1,\ldots,X_n$ be independent and identically distributed copies of a random variable $X$ taking values in $\mathbb{N}$ and denote the distribution of $X$ by $\mathbb{P}^X$. The testing problem of interest is to test the composite hypothesis
\begin{equation}\label{eq:H0}
    H_0:\,\mathbb{P}^X\in\mathcal{Z}
\end{equation}
against general alternatives based on the sample $X_1,\ldots,X_n$.

The existing literature on this testing problem is relatively limited. General omnibus procedures based on the empirical distribution function, such as Kolmogorov–Smirnov and Cramér-von Mises type tests proposed in \cite{H:1996}, can be adapted to the present setting. Kernelized Stein discrepancy tests for discrete sample spaces are developed in \cite{Yang:2018}, but they are formulated only for simple null hypotheses on finite sample spaces, and therefore require suitable modifications to address the composite hypothesis in \eqref{eq:H0}. In \cite{BEN:2022}, the authors propose general tests derived from characterizations of probability mass functions, which are also applicable to this problem. To the best of our knowledge, the only procedure specifically tailored to the zeta family of distributions is that of \cite{M:2009}, where the empirical Mellin transform of inverse moments is incorporated into a weighted $L^2$-type statistic.

We propose a new competitor to \cite{M:2009} based on a Stein characterization of the Zeta law using the generator approach, see \cite{B:1988,B:1990}, by providing a birth and death process whose stationary distribution is precisely the Zeta$(s)$ distribution and which then induces using the infinitesimal generator a suitable Stein operator. Note that this is the first goodness-of-fit test that uses the generator approach to provide the Stein characterization and then we follow \cite{Aetal:2023}, Section 5.2, to propose a weighted $L^2$-type test statistic. Firstly, consider a birth–death process \( (Y_t)_{t\ge0} \) on \( \mathbb{N} \) with birth and death rates
\[
\lambda_k = 1, \qquad
\mu_k = \left(\frac{k}{k-1}\right)^{s}, \quad k \ge 2,
\]
and \(\mu_1 = 0\). These rates satisfy the detailed balance condition
\[
p_s(k)\lambda_k = p_s(k+1)\mu_{k+1}, \qquad k\ge1,
\]
so that \(P_s= \mathrm{Zeta}(s)\) is the unique stationary (and reversible) distribution of this Markov chain. The infinitesimal generator $L_s$ of the process acts on functions \( f:\mathbb{N}\to\mathbb{R} \) as
\begin{equation}\label{eq:SOP}
L_s f(k)
= \lambda_k\big(f(k+1)-f(k)\big)
+ \mu_k\big(f(k-1)-f(k)\big), \quad k\ge1,
\end{equation}
with the convention that the second term vanishes for \(k=1\), i.e.
\[
L_s f(1) = f(2) - f(1),
\qquad
L_s f(k)
= \big(f(k+1) - f(k)\big)
+ \left(\frac{k}{k-1}\right)^{s}\big(f(k-1) - f(k)\big), \quad k \ge 2.
\]
Following the generator approach \citep{B:1988,B:1990}, this infinitesimal generator is a Stein operator for the stationary distribution, see also Section 1 in \cite{ER:2008}. In our case this leads to the following characterization of the Zeta distribution.
\begin{theorem}\label{thm:char1}
Let $X \sim P_s = \mathrm{Zeta}(s)$, for some $s>0$. Then we have for all functions \( f:\mathbb{N}\to\mathbb{R} \) such that the expectations below exist,
\[
\mathbb{E}_{P_s}\big[L_s f(X)\big] = 0.
\]
Conversely, if a probability measure \(P\) on \(\mathbb{N}\) satisfies
\(\mathbb{E}_P[L_s f(X)] = 0\) for all bounded \(f\),
then \(P = P_s\).
\end{theorem}

\begin{proof}
The forward direction follows from the reversibility condition \(p_s(k) = p_s(k+1)\mu_{k+1}\),
which ensures that for all \(f\),
\[
\mathbb{E}_{P_s}\big[L_s f(X)\big]=\sum_{k=1}^\infty p_s(k) L_s f(k)
= \sum_{k=1}^\infty p_s(k)\lambda_k (f(k+1)-f(k))
  + \sum_{k=2}^\infty p_s(k)\mu_k (f(k-1)-f(k)) = 0.
\]
The converse follows from the fact that the stationary distribution of a positive recurrent birth–death process is unique.
\end{proof}
\begin{corollary}\label{cor:char}
For the family of test functions $\{f_t(x) = t^x: t\in[0,1]\}$ and $s>1$ we have
\begin{equation}\label{eq:SOP_PGF}
L_s f_t(1) = t^2 - t,
\qquad
L_s f_t(k)
= \left(t^{k+1}-t^k\right)
+ \left(\frac{k}{k-1}\right)^{s}(t^{k-1}-t^k),
\quad k \ge 2,
\end{equation}
and the class of test functions $\{f_t: t\in[0,1]\}$ is characterizing, i.e. $\mathbb{E}\left[L_s f_t(X)\right] = 0$ for all $t\in[0,1]$ if and only if $X\sim P_s$.
\end{corollary}
\begin{proof}
We apply \eqref{eq:SOP} to $f_t(x)=t^x$, $t\in[0,1]$, which yields directly \eqref{eq:SOP_PGF}. To show that $\{f_t:t\in[0,1]\}$ is characterizing, let $q$ be any probability mass function (pmf) on $\mathbb{N}$ satisfying
$\mathbb{E}_q[(L_s f_t)(X)]=0$ for all $t\in[0,1]$.
Substituting the above expression gives
\[
0=\sum_{k\ge1} q(k)\left(t^{k+1}-t^k + \left(\frac{k}{k-1}\right)^s \left(t^{k-1}-t^k\right)\right).
\]
Define  the generating functions
$G_q(t)=\sum_{k\ge1} q(k)t^k$,
$H_q(t)=\sum_{k\ge1} q(k)\mu_k t^k$.
Then for all $t\in[0,1]$,
\[
0=(1-t)\big(H_q(t)-tG_q(t)\big),
\]
implying $H_q(t)=tG_q(t)$ as an identity of power series.
Matching coefficients gives $q(k)\mu_k=q(k-1)$ for $k\ge2$, hence
\[
q(k)=q(1)\prod_{j=2}^k \frac{1}{\mu_j}
= q(1)\prod_{j=2}^k \left(\frac{j-1}{j}\right)^s
= q(1)k^{-s}.
\]
Normalization yields $q(1)=1/\zeta(s)$, i.e.\ $q(k)=k^{-s}/\zeta(s)$.
Thus the family $\{f_t:t\in[0,1]\}$ uniquely characterizes the $\mathrm{Zeta}(s)$ law.
\end{proof}

Given i.i.d.\ data \(X_1,\dots,X_n\) from an unknown distribution \(P\), we define
the empirical Stein process as
\[
Z_n(t;s) = \frac{1}{\sqrt{n}} \sum_{i=1}^n L_s f_t(X_i), \qquad t\in[0,1].
\]
and to test the composite hypothesis in \eqref{eq:H0} we propose the test statistic
\begin{equation}\label{eq:test_statistic}
T_n = \int_0^1 |Z_n(t;\widehat{s}_n)|^2 w(t)\,{\rm d}t,
\end{equation}
where $\widehat{s}_n$ is a consistent estimator of $s$ and $w$ is a positive weight function satisfying $\int_0^1w(t){\rm d}t<\infty$. Since under the null hypothesis $Z_n$ is close to zero, $T_n$ should produce small values and hence we reject $H_0$ for 'large' values of the test statistic. Since the limit null distribution of $T_n$ clearly depends on the true but unknown value $s_0$ (say) of the Zeta distribution, a resampling technique is necessary to compute suitable critical values.

For  the weight function $w(t)=(1-t)^\beta$, $\beta\ge0$, we have the explicit and integration free representation
\begin{eqnarray*}
T_{n,\beta}&=&\frac2n\sum_{j,k=1}^n\frac{\mathbf{1}\{X_{n,j}=X_{n,k}=1\}}{(3+\beta)(4+\beta)(5+\beta)}\\&&-2\left[\left(\frac{X_{n,j}}{X_{n,j}-1}\right)^{\widehat{s}_n}-\frac{X_{n,j}+1}{\beta+X_{n,j}+4}\right]B(X_{n,j}+1,3+\beta)\mathbf{1}\{X_{n,j}\ge2,X_{n,k}=1\}\\
&&+\left[B(X_{jk}^+-1,3+\beta)\left(\frac{X_{n,j}}{X_{n,j}-1}\right)^{\widehat{s}_n}\left(\frac{X_{n,k}}{X_{n,k}-1}\right)^{\widehat{s}_n}\right.
\\&&\left.-B(X_{jk}^+,3+\beta)\left(\left(\frac{X_{n,j}}{X_{n,j}-1}\right)^{\widehat{s}_n}+\left(\frac{X_{n,k}}{X_{n,k}-1}\right)^{\widehat{s}_n}\right)+B(X_{jk}^++1,3+\beta))\right]\mathbf{1}\{X_{n,j},X_{n,k}\ge2\},
\end{eqnarray*}
where $\mathbf{1}$ denotes the indicator function, $B(\cdot,\cdot)$ is the Beta-function and $X_{jk}^+=X_{n,j}+X_{n,k}$.

The remainder of the paper is structured as follows. In the next section we derive the asymptotic properties of the test statistic by applying a central limit theorem for triangular arrays of Hilbert-space-valued random elements, and we establish both the consistency of the test and the validity of the proposed parametric bootstrap scheme. Section~\ref{sec:Simu} presents a comprehensive Monte Carlo study for the testing problem in \eqref{eq:H0}, and Section~\ref{sec:out} concludes with an outlook and several directions for future research.

\section{Asymptotics}\label{sec:Asy}

In this section we derive asymptotic properties of the proposed tests. Let \(L_{w}^2=L^2([0,1],\mathcal{B}_{[0,1]}, w(t){\rm d}t)\) be the separable Hilbert space (of equivalence classes) of Borel-measurable functions \(g:[0,1] \rightarrow \R\) satisfying $\|g\|_{L^2_w}^2 = \int_{0}^{1}g^2(t)w(t){\rm d}t < \infty$ with respect to the measurable positive weight function $w(\cdot)$. The scalar product on $L_w^2$ is defined by \(\langle g,h\rangle = \int_{0}^{\infty} g(t)h(t)w(t){\rm d}t\). This is a suitable setting, since the test statistic admits the representation
\begin{equation*}
T_n=\|Z_n(\cdot;\widehat{s}_n)\|_{L^2_w}^2.
\end{equation*}

\subsection{Limit null distribution}
Let $X_{n,1},\ldots,X_{n,n}$ be a triangular array of row-wise independent and identically distributed copies of $X_n\sim P_{s_n}$, where $(s_n)$ is a sequence of numbers with $s_n>1$ and $\lim_{n\to\infty}s_n=s_0>1$.
We choose as consistent estimator the maximum-likelihood estimator (MLE) $\widehat{s}_n$,
which satisfies
\[
\frac{\zeta'(\widehat{s}_n)}{\zeta(\widehat{s}_n)} = - \frac{1}{n}\sum_{i=1}^n \log X_i.
\]
Since it produces dependencies in the sum of the empirical Stein process, we provide the Bahadur representation of the estimator and related helping processes that share the same limit distribution. Denote by \(u(x;s) = \partial_s \log p_s(x)=-\log(x)-\zeta'(s)/\zeta(s)\), $s>1$, the Zeta score function and by \(I(s) = \mathrm{Var}(u(X;s))=\zeta''(s)/\zeta(s)-\left(\zeta'(s)/\zeta(s)\right)^2\), $s>1$, the Fisher information. Then applying the Bahadaur representation, see Corollary 10.16 in \cite{H:2024}, we get
\begin{eqnarray}\nonumber
    \sqrt{n}(\widehat{s}_n-s_0)&=&\frac{I(s)^{-1}}{\sqrt{n}}\sum_{j=1}^nu(X_j;s_0) +o_{\mathbb{P}}(1)\\
    &=&-\frac{1}{\sqrt{n}\left(\zeta''(s_0)/\zeta(s_0)-\left(\zeta'(s_0)/\zeta(s_0)\right)^2\right)}\sum_{j=1}^n\left(\log(X_j)+ \frac{\zeta'(s_0)}{\zeta(s_0)}\right)+o_{\mathbb{P}}(1).\label{eq:BRMLE}
\end{eqnarray}
Define the functions
\begin{equation}\label{eq:hs}
    h_s(x,t)=-t(1-t)\mathbf{1}\{x=1\}+(1-t)t^{x-1}\left[\left(\frac{x}{x-1}\right)^{s}-t\right]\mathbf{1}\{x\ge2\},\quad t\in[0,1],
\end{equation}
and
\begin{equation}
    \frac{\partial h_s(x,t)}{\partial s}=(1-t)t^{x-1}\left(\frac{x}{x-1}\right)^{s}\log\left(\frac x{x-1}\right)\mathbf{1}\{x\ge2\}=:g_s(x,t),\quad t\in[0,1].
\end{equation}
Now, consider the helping processes
\begin{equation}
\widetilde{Z}_n(t;s_0)=\frac1{\sqrt{n}}\sum_{j=1}^nh_{s_0}(X_{n,j},t)+g_{s_0}(X_{n,j},t)(\widehat{s}_n-s_0),\quad t\in[0,1],
\end{equation}
and
\begin{equation}
\widehat{Z}_n(t;s_0)=\frac1{\sqrt{n}}\sum_{j=1}^nh_{s_0}(X_{n,j},t)+\mathbb{E}\left[g_{s_0}(X_{n,1},t)\right]I(s_0)^{-1}u(X_{n,j};s_0),\quad t\in[0,1].
\end{equation}

\begin{lemma}\label{lem:asy_equiv}
Under the standing assumptions, we have
\begin{equation*}
\left\|Z_n(\cdot;\widehat{s}_n)-\widetilde{Z}_n(\cdot;s_0)\right\|_{L^2_w}=o_{\mathbb{P}}(1)\quad \mbox{and}\quad \left\|\widetilde{Z}_n(\cdot;s_0)-\widehat{Z}_n(\cdot;s_0)\right\|_{L^2_w}=o_{\mathbb{P}}(1).
\end{equation*}
\end{lemma}

\begin{proof}
To prove the first statement, write
\[
  Z_n(t;\widehat{s}_n)-\widetilde Z_n(t;s_0)
  = \frac{1}{\sqrt{n}} \sum_{j=1}^n
     R_n(X_{n,j},t),
\]
where $R_n(x,t)= h_{\widehat s_n}(x,t)-h_{s_0}(x,t)-g_{s_0}(x,t)(\widehat s_n-s_0)$.
For each fixed $(x,t)$ the map $s\mapsto h_s(x,t)$ is $C^2$. A second order
Taylor expansion around $s_0$ gives
\[
  R_n(x,t)
  = \frac12(\widehat{s}_n-s_0)^2\,\partial_s^2 h_{\theta_n(x)}(x,t),
\]
for some $\theta_n(x)$ between $s_0$ and $\widehat{s}_n$. For $x\ge2$ we have
\[
  \frac{\partial h_s(x,t)}{\partial_s^2}
  = (1-t)t^{x-1}\Big(\frac{x}{x-1}\Big)^s
      \Bigl[\log\Bigl(\frac{x}{x-1}\Bigr)\Bigr]^2\mathbf 1_{\{x\ge2\}},
\]
while for $x=1$ this derivative is zero since $h_s(1,t)=-t(1-t)$ does not
depend on $s$. Hence there exists $\delta>0$ and a finite constant $C$ such
that, on an event of probability tending to one,
\[
  \bigl|R_n(x,t)\bigr|
  \le C(\widehat{s}_n-s_0)^2 B(x,t),
\]
with
\[
  B(x,t)
  = (1-t)t^{x-1}\Big(\frac{x}{x-1}\Big)^{s_0+\delta}
       \Bigl[\log\Bigl(\frac{x}{x-1}\Bigr)\Bigr]^2\mathbf 1_{\{x\ge2\}}.
\]
By inspection of the tails of the Zeta distribution and the behaviour
$\log\!\bigl(\frac{x}{x-1}\bigr)\sim 1/(x-1)$ as $x\to\infty$, one checks
that $\mathbb{E}\,\bigl\|B(X_{n,1},\cdot)\bigr\|_{L_w^2}<\infty$,
uniformly in $n$. Therefore by the strong law of large numbers in Hilbert spaces $\frac{1}{n}\sum_{j=1}^n \bigl\|B(X_{n,j},\cdot)\bigr\|_{L_w^2}$ converges almost surely to $\mathbb{E}\,\bigl\|B(X_{n,1},\cdot)\bigr\|_{L_w^2}$.

Using the triangle inequality and the bound for $R_n$,
\[
  \bigl\|Z_n(\cdot;\widehat{s}_n)-\widetilde Z_n(\cdot;s_0)\bigr\|_{L_w^2}
  \le C\sqrt{n}(\widehat{s}_n-s_0)^2 \frac{1}{n}
        \sum_{j=1}^n \bigl\|B(X_{n,j},\cdot)\bigr\|_{L_w^2}
\]
By the Bahadur representation \eqref{eq:BRMLE},
$\sqrt{n}\left(\widehat{s}_n-s_0\right) = O_{\mathbb{P}}(1)$, hence since $\widehat{s}_n-s_0=o_{\mathbb{P}}(1)$ we have $\sqrt{n}(\widehat{s}_n-s_0)^2=o_{\mathbb{P}}(1)$. Thus
\[
  \bigl\|Z_n(\cdot;\widehat{s}_n)-\widetilde Z_n(\cdot;s_0)\bigr\|_{L_w^2}
  = o_{\mathbb{P}}(1).
\]

\smallskip
For the second statement, set
\[
  U_n = \frac{1}{\sqrt{n}}\sum_{j=1}^n u(X_{n,j};s_0),\qquad
  G_n(t) = \frac{1}{n}\sum_{j=1}^n g_{s_0}(X_{n,j},t),\qquad
  \bar g(t) = \mathbb{E}\,g_{s_0}(X_{n,1},t).
\]
Then
\[
  \frac{1}{\sqrt{n}}\sum_{j=1}^n g_{s_0}(X_{n,j},t)
  = \sqrt{n}\,G_n(t),
\]
and
\begin{align*}
  \widetilde Z_n(t;s_0)-\widehat Z_n(t;s_0)
   &= (\widehat{s}_n-s_0)\sqrt{n}\,G_n(t)
      - I(s_0)^{-1}\bar g(t)\,U_n \\
   &= \Bigl[\sqrt{n}(\widehat{s}_n-s_0)-I(s_0)^{-1}U_n\Bigr]G_n(t)
      + I(s_0)^{-1}U_n\bigl(G_n(t)-\bar g(t)\bigr).
\end{align*}
Consequently,
\[
  \bigl\|\widetilde Z_n(\cdot;s_0)-\widehat Z_n(\cdot;s_0)\bigr\|_{L_w^2}
  \le A_n + B_n,
\]
where
\[
  A_n = \Bigl|\sqrt{n}(\widehat{s}_n-s_0)-I(s_0)^{-1}U_n\Bigr|\,
         \|G_n\|_{L_w^2},
  \qquad
  B_n = |I(s_0)^{-1}U_n|\,
         \|G_n-\bar g\|_{L_w^2}.
\]

By the Bahadur representation~\eqref{eq:BRMLE},
$\sqrt{n}(\widehat{s}_n-s_0)-I(s_0)^{-1}U_n=o_{\mathbb{P}}(1)$.
Moreover, $g_{s_0}(X_{n,1},\cdot)\in L_w^2$ with finite second moment, so
by the law of large numbers in the Hilbert space $L_w^2$,
\[
  G_n \xrightarrow{\mathbb{P}} \bar g
  \quad\text{in }L_w^2,
\]
which implies $\|G_n\|_{L_w^2}=O_{\mathbb{P}}(1)$ and
$\|G_n-\bar g\|_{L_w^2}\to0$ in probability. Finally, $U_n$ is a
normalized sum of i.i.d.\ variables with variance $I(s_0)$, so
$U_n=O_{\mathbb{P}}(1)$. Hence $A_n = o_{\mathbb{P}}(1)= B_n$ and therefore $\bigl\|\widetilde Z_n(\cdot;s_0)-\widehat Z_n(\cdot;s_0)\bigr\|_{L_w^2} = o_{\mathbb{P}}(1)$.
\end{proof}

For the statement in the next theorem we write
\[
a(t) = \mathbb E_{H_0}g_{s_0}(X_1,t)= \frac{1-t}{\zeta(s_0)} \sum_{k=2}^\infty t^{k-1} \left( \frac{k}{k-1} \right)^{s_0} k^{-s_0} \log\left( \frac{k}{k-1} \right),\quad t\in[0,1].
\]
\begin{theorem}\label{thm:limitnull}
Under the triangular array stated at the beginning of this section, we have as \( n \to \infty \)
\[
T_n \xrightarrow{d} \| Z \|^2_{L^2_w},
\]
where \( Z(\cdot) \) is a centered Gaussian process in $L_w^2$ with covariance kernel
\begin{eqnarray*}
C(s, t) &=& \frac{st (1-s)(1-t)}{\zeta(s_0)}- \frac{a(s)a(t)}{I(s_0)}\\&&
+ \frac{(1-s)(1-t)}{\zeta(s_0)} \sum_{k=2}^\infty k^{-s_0} (st)^{k-1} \left( \left(\frac k {k-1}\right)^{s_0} - s \right)\left( \left(\frac k {k-1}\right)^{s_0} - t \right) , \quad s,t\in[0,1].
\end{eqnarray*}
\end{theorem}
\begin{proof}
By Lemma \ref{lem:asy_equiv} and the triangular inequality the limit distribution of the three processes is the same, hence it is enough to focus on $\widehat{Z}_n$. Obiously, we have a sum of row-wise iid. random variables so we apply the central limit theorem for triangular arrays in Hilbert spaces, see Theorem 17.30 in \cite{H:2024}. Define the $L_w^2$-valued random
elements
\[
  Y_{n,j}(t)
  = \frac1{\sqrt n}\left(h_{s_n}(X_{n,j},t)+a(t)I(s_n)^{-1}u(X_{n,j};s_n)\right),
  \qquad t\in[0,1],
\]
where under $H_0$ the $X_{n,j}$ are iid. with law $P_{s_n}$. Then
\[
  \widehat Z_n(\cdot;s_0)
  = \sum_{j=1}^n Y_{n,j}.
\]
By construction and by the identities
$\mathbb E_{s_n}h_{s_n}(X_{n,1},t)=0$ and
$\mathbb E_{s_n}u(X_{n,1};s_n)=0$, we have
$\mathbb E_{s_n}Y_{n,j}=0\in L^2_w$.
From the explicit formulas for $h_s$ and $u(\cdot;s)$ and the tail behaviour
of the Zeta($s$) distribution, one checks exactly as in Lemma \ref{lem:asy_equiv} that
\[
  \sup_{n\ge1}\mathbb E_{s_n}\|h_{s_n}(X_{n,1},\cdot)\|^2<\infty,
  \qquad
  \sup_{n\ge1}\mathbb E_{s_n}u(X_{n,1};s_n)^2<\infty,
\]
and hence
\[
  \sup_{n\ge1}\mathbb E_{s_n}\|Y_{n,1}\|^2<\infty.
\]
Thus the triangular array $\{Y_{n,j}\}$ satisfies the basic assumptions of
Theorem 17.30 in \cite{H:2024}.

Let $S_n=\sum_{j=1}^n Y_{n,j}=\widehat Z_n(\cdot;s_n)$ and let $C_n$ be the
covariance operator of $S_n$ in $L^2_w$. Because the $Y_{n,j}$ are
independent and centred,
\[
  C_n
  = \sum_{j=1}^n \operatorname{Cov}(Y_{n,j})
  = n\,\operatorname{Cov}(Y_{n,1}).
\]
Writing
\[
  W_{s_n}(t)
  = h_{s_n}(X_{n,1},t)+a(t)I(s_n)^{-1}u(X_{n,1};s_n),
\]
we have $Y_{n,1}=W_{s_n}/\sqrt n$, so
\[
  C_n = \operatorname{Cov}(W_{s_n})\qquad(n\ge1).
\]

Fix an orthonormal basis $\{e_k\}_{k\ge1}$ of $L^2_w$ and set $a_{k,\ell}
   = \lim_{n\to\infty} \langle C_n e_k,e_\ell\rangle$, $k,\ell\ge1$,
whenever the limit exists. For each $k,\ell$,
\[
  \langle C_n e_k,e_\ell\rangle
  = \mathbb E_{s_n}\big[\langle W_{s_n},e_k\rangle
                       \langle W_{s_n},e_\ell\rangle\big].
\]
As $s_n\to s_0$ and the maps $(s,x)\mapsto h_s(x,\cdot)$ and
$(s,x)\mapsto u(x;s)$ are continuous, we have
\[
  \langle W_{s_n},e_k\rangle
  \xrightarrow{\;\;}\langle W_{s_0},e_k\rangle
  \quad\text{a.s.},
\]
where $W_{s_0}(t)=h_{s_0}(X_1,t)+a(t)I(s_0)^{-1}u(X_1;s_0)$ and
$X_1\sim P_{s_0}$. Moreover by Cauchy-Schwarz
\[
  \big|\langle W_{s_n},e_k\rangle
        \langle W_{s_n},e_\ell\rangle\big|
  \le \frac12\big(\langle W_{s_n},e_k\rangle^2
                  +\langle W_{s_n},e_\ell\rangle^2\big)
  \le \|W_{s_n}\|^2,
\]
and $\sup_n \mathbb E\|W_{s_n}\|^2<\infty$ by the moment bounds
above. Dominated convergence yields
\[
  \langle C_n e_k,e_\ell\rangle
   \longrightarrow
   \mathbb E_{s_0}\big[\langle W_{s_0},e_k\rangle
                      \langle W_{s_0},e_\ell\rangle\big]
   = a_{k,\ell},
\]
so assumption (a) of Theorem 17.30 in \cite{H:2024} holds. Further,
\[
  \sum_{k=1}^\infty a_{k,k}
  = \mathbb E_{s_0}\|W_{s_0}\|^2<\infty,
\]
which is assumption (b).

For assumption (c), fix an $\epsilon>0$ then the quantity
$L_n(\varepsilon,e_k)$ in Theorem 17.30 in \cite{H:2024} becomes
\[
  L_n(\varepsilon,e_k)
  = \sum_{j=1}^n
      \mathbb E\Big[\langle Y_{n,j},e_k\rangle^2
        \mathbf 1{\{|\langle Y_{n,j},e_k\rangle|>\varepsilon\}}\Big]
  = n\,\mathbb E\Big[\langle Y_{n,1},e_k\rangle^2
        \mathbf 1{\{|\langle Y_{n,1},e_k\rangle|>\varepsilon\}}\Big].
\]
Since $\langle Y_{n,1},e_k\rangle
   = \langle W_{s_n},e_k\rangle/\sqrt n$, we can rewrite
\[
  L_n(\varepsilon,e_k)
  = \mathbb E\Big[\langle W_{s_n},e_k\rangle^2\,
       \mathbf 1{\{|\langle W_{s_n},e_k\rangle|>\varepsilon\sqrt n\}}\Big]
  \xrightarrow[n\to\infty]{} 0
\]
by dominated convergence and the moment bound
$\sup_n\mathbb E\langle W_{s_n},e_k\rangle^2<\infty$.
Thus assumption (c) is also satisfied.

By Theorem 17.30 in \cite{H:2024} there exists a centered Gaussian element
$Z\in\mathbb H$ with covariance operator $C$ determined by
\[
  \langle Cx,e_\ell\rangle
   = \sum_{k=1}^\infty a_{k,\ell}\langle x,e_k\rangle,
  \qquad x\in\mathbb H,\ \ell\ge1,
\]
such that
\[
  \widehat Z_n(\cdot;s_n) = S_n \xrightarrow{d} Z
  \quad\text{in }L_w^2.
\]

By construction,
\[
  C(s,t)
  = \mathbb E_{s_0}\big[W_{s_0}(s)\,W_{s_0}(t)\big],
  \qquad s,t\in[0,1].
\]
With
\[
  W_{s_0}(t)=h_{s_0}(X_1,t)+a(t)I(s_0)^{-1}u(X_1;s_0),
\]
and using $\mathbb E_{s_0}h_{s_0}(X_1,t)=0$,
$\mathbb E_{s_0}u(X_1;s_0)=0$, we obtain
\[
  C(s,t)
  = \mathbb E_{s_0}\big[h_{s_0}(X_1,s)h_{s_0}(X_1,t)\big]
    + a(s)a(t)I(s_0)^{-1}
    + a(s)\,\mathrm{Cov}\!\big(u(X_1;s_0),h_{s_0}(X_1,t)\big)
    + a(t)\,\mathrm{Cov}\!\big(u(X_1;s_0),h_{s_0}(X_1,s)\big).
\]
The Stein identity implies $\mathbb E_s h_s(X_1,t)=0$ for all $s>1$ and
$t\in[0,1]$. Differentiating at $s=s_0$ yields
\[
  0
  = \mathbb E_{s_0}g_{s_0}(X_1,t)
    + \mathbb E_{s_0}\big[h_{s_0}(X_1,t)\,u(X_1;s_0)\big]
  = a(t)+\mathrm{Cov}\!\big(u(X_1;s_0),h_{s_0}(X_1,t)\big),
\]
so $\mathrm{Cov}(u(X_1;s_o),h_{s_0}(X_1,t))=-a(t)$. Inserting this into the last
display gives
\[
  C(s,t)
  = \mathbb E_{s_0}\big[h_{s_0}(X_1,s)h_{s_0}(X_1,t)\big]
    - \frac{a(s)a(t)}{I(s_0)}.
\]
Finally, evaluating the expectation explicitly under the Zeta($s_0$)
distribution using the definition of $h_{s_0}$ in \eqref{eq:hs} leads, after a
straightforward calculation, to the series representation for $C(s,t)$
stated in the theorem.

By Lemma \ref{lem:asy_equiv}, $\|Z_n(\cdot;\widehat{s}_n)-\widehat Z_n(\cdot;s_n)\|_{L_w^2}=o_{\mathbb{P}}(1)$. Hence $Z_n(\cdot;\widehat{s}_n)\xrightarrow{d} Z$ in $L_w^2$, and the
continuous mapping theorem yields
\[
  T_n = \big\|Z_n(\cdot;\widehat{s}_n)\big\|_{L_w^2}^2
  \xrightarrow{d} \|Z\|_{L_w^2}^2,
\]
which concludes the proof.
\end{proof}

\subsection{Consistency of the testing procedure}
We now prove that the test which rejects the hypothesis $H_0$ for large values of $T_n$ is consistent against general alternatives. Hereafter, we consider an iid. sequence $(X_n)_{n\in \N}$ of copies of $X$, where  \(X\) is a non-degenerate positive random variable satisfying \(\mathbb{E}[X^2] < \infty\). Moreover, we assume that there is \(s_0>0\) such that
\begin{align}
\widehat{s}_{n} \stackrel{\text{a.s.}}{\longrightarrow} s_0, \qquad \text{as} \quad n \rightarrow \infty,
\label{consistency_convergence_estimators}
\end{align}
where \(\widehat{s}_{n}\) is the maximum likelihood estimators as in the previous section. The following result is a direct consequence of a Taylor expansion and Fatou's lemma.
\begin{theorem}\label{thm:consistency}
Under the stated conditions, we have
\begin{align*}
\liminf_{n \rightarrow \infty} \frac{T_n}{n} \geq \Lambda_{s_0,w} \qquad \mathbb{P} \textnormal{-a.s.},
\end{align*}
where
\begin{align*}
\Lambda_{s_0,w}= \int_{0}^{1} \mathbb{E}\big[h_{s_0}(X,t)\big]^2 w(t){\rm d}t.
\end{align*}
\end{theorem}

\begin{remark}
Under a fixed alternative, the MLE $\widehat{s}_n$ converges almost surely to $s_0$ if the limit
\[
s_0=\arg\max_{s>1}\mathbb{E}\left[\log p_s(X)\right]
=\arg\max_{s>1}\mathbb{E}\left[-s\log X-\log\zeta(s)\right]
\]
exists and is unique, where $X$ follows the true distribution $P$.
This holds when $\mathbb{E}[|\log X|]<\infty$ and the Kullback--Leibler divergence between
$P$ and $\mathrm{Zeta}(s)$ is minimized at a unique $s_0>1$.
\end{remark}
\subsection{Parametric Bootstrap procedure}\label{subsec:Boot}
Since the limit null distribution of $T_n$ in Theorem \ref{thm:limitnull} depends on the unknown parameter $s_0>1$ of the underlying
Zeta distribution, we propose a parametric bootstrap procedure to obtain critical values.
For a sample $X_1,\ldots,X_n$ satisfying the assumptions above, we compute the MLE
$\widehat{s}_n=\widehat{s}_n(X_1,\ldots,X_n)$.

We then generate a bootstrap sample of size $n$, say $X_1^*,\ldots,X_n^*$, following the
$\mathrm{Zeta}(\widehat{s}_n)$ distribution, estimate the parameter $s$ from $X_1^*,\ldots,X_n^*$
(denote it $\widehat{s}_n^*$), and calculate the test statistic $T_n^*$.

By repeating this procedure $b$ times, we obtain $T_{n,1}^*,\ldots,T_{n,b}^*$ and compute
the empirical distribution function
\[
H_{n,b}^*(t)=\frac{1}{b}\sum_{i=1}^b\mathbf{1}\{T_{n,i}^*\le t\},\qquad t\ge 0.
\]

Given the nominal level $\alpha\in[0,1]$, we use the empirical $(1-\alpha)$-quantile:
\[
c_{n,b}^*(\alpha)=H_{n,b}^{*-1}(1-\alpha)=\begin{cases}
T_{b(1-\alpha):b}^*, & b(1-\alpha)\in\mathbb{N},\\[0.2em]
T_{\lfloor b(1-\alpha)\rfloor+1:b}^*, & \text{otherwise,}
\end{cases}
\]
where $T_{1:b}^*,\ldots,T_{b:b}^*$ are the order statistics. We reject $H_0$ if $T_n>c_{n,b}^*(\alpha)$.

Denote the distribution function of $T_n$ under $\mathrm{Zeta}(s)$ by
$H_{n,s}(t)=\mathbb{P}_s(T_n\le t)$, and the limit distribution by
$H_s(t)=\mathbb{P}(\|Z\|_{L^2_w}^2\le t)$, where $Z$ is the centered Gaussian element
from the limit null distribution.

The function $H_s$ is continuous and strictly monotone. By consistency of $\widehat{s}_n$
and continuity of $H_s$, for each $t\ge 0$,
\[
\lim_{n\to\infty}H_{n,\widehat{s}_n}(t)=H_{s_0}(t)\quad\mathbb{P}\text{-a.s.}
\]

By a triangular version of the bootstrap argument (as in \cite{H:1996}, Theorem 3.6), we have
\[
\sup_{t\ge 0}\left|H_{n,b}^*(t)-H_{n,\widehat{s}_n}(t)\right|\xrightarrow{\mathbb{P}}0
\quad\text{as }b,n\to\infty.
\]

Thus, $c_{n,b}^*(\alpha)\xrightarrow{\mathbb{P}}H_{n,\widehat{s}_n}^{-1}(1-\alpha)$ as $b\to\infty$.

If $X_1,\ldots,X_n$ are i.i.d.\ from $\mathrm{Zeta}(s_0)$, the continuity of $H_{s_0}$ yields
\[
\lim_{n\to\infty}\lim_{b\to\infty}\mathbb{P}(T_n>c_{n,b}^*(\alpha))=\alpha.
\]

If $X_1,\ldots,X_n$ do not follow a Zeta distribution, then by Theorem 2.3,
$\Lambda_{s_0,w}>0$, so
\[
\lim_{n\to\infty}\lim_{b\to\infty}\mathbb{P}(T_n>c_{n,b}^*(\alpha))=1.
\]

Thus the test is consistent against any fixed alternative distribution satisfying
the stated assumptions.

\begin{remark}\label{rem:KLE}
    Since the limit random element in Theorem \ref{thm:limitnull} is a centered Gaussian process taking values in $L^2_w$ with known covariance kernel $C(s,t)$, $s,t\in[0,1]$ it is well-known by the Karhunen-Loève expansion (see \cite{DW:2014}, Chapter 3), that the distribution of $\|Z\|^2_{L^2_w}$ corresponds to $\sum_{j=1}^\infty \lambda_j(s_0) N_j^2$, where $N_j$ are iid. standard normal random variables and $(\lambda_j(s_0))_{j\in\N}$ is a sequence of decaying positive eigenvalues defined by the (linear second-order homogeneous Fredholm) integral equation
\begin{equation*}\label{int:eq}
\lambda f(s) = \mathcal{C} f(s), \quad s \in [0,1],
\end{equation*}
where $\mathcal{C}: L^2_w \mapsto L^2_w$ is defined as
\[
\mathcal{C} f(s) = \int_0^1 C(s,t) f(t) w(t)\mbox{d}t,
\]
corresponding to the covariance kernel $C$ of $Z$ in Theorem \ref{thm:limitnull}. The eigenvalues obviously depend on the true unknown parameter $s_0>1$ as well as the weight function $w(\cdot)$. By the complexity of the kernel $C$ it seems hopeless to find an analytical solution to the eigenvalue problem, although numerical procedures as the Rayleigh-Ritz method could be applied \citep{EJM:2025}.
\end{remark}

\begin{remark}
    Defining the feature map $\Phi_s(x)(t)=h_s(x,t)$ and the Stein kernel $k_s(x,y)=\langle\Phi_s(x),\Phi_s(y)\rangle_{L^2_w}$, we have for
    $T_n$ in \eqref{eq:test_statistic} a $V$-statistic $T_n=n^{-1}\sum_{i,j=1}^nk_{\widehat{s}_n}(X_j,X_k)$, which shows that $T_n$ is a special maximum mean discrepancy (MMD)-type statistic with a Stein-type kernel, and because $k_s$ is defined as an inner product of features, it is automatically positive definite. For more details on this type of test statistics see \cite{KGBF:2025} and for related theory \cite{BRB:2025}.
\end{remark}

\section{Simulations}\label{sec:Simu}

In this section we provide (to the best of our knowledge first) comparative Monte Carlo Simulation study for testing the fit to the Zeta family of distributions. In each simulation run, we simulate a data set $X_1,\ldots,X_n$ of sample size $n=100$, and fix the significance level to $5\%$. Every entry in Table \ref{tab:100} is based on 10000 replications. To ensure an 'apples-to-apples' comparison, we employ for all the considered testing procedures the same resampling procedure. Since the parametric bootstrap in Section \ref{subsec:Boot} is computationally intensive, we employ the warp-speed
Monte Carlo method of \citet{GPW:2013}. Hence, for each simulated dataset we compute
the test statistic on the original sample and on one bootstrap resample, and then use the
resulting pairs to estimate rejection probabilities as the number of Monte Carlo
replications grows.

There are relatively few goodness-of-fit tests specifically designed for
nonnegative integer-valued random variables. We adopt the approaches mentioned in the introduction to
compare the finite sample performance of the test statistics.  \citet{H:1996} proposed a test
statistic for parametric families of discrete distributions, based on a
comparison between the empirical distribution function and the cumulative
distribution function with estimated parameters.

Let us briefly describe the test of \citet{H:1996}. We present it here in its
specific form for the Zeta distribution with unknown parameter $s$. Let
$X_1,\ldots,X_n$ be an i.i.d.\ sequence of random variables following the
Zeta$(s)$ distribution, and let $\widehat{F}_n$ denote the empirical
distribution function based on this sample. Furthermore, let $\widehat{s}_n$ be
a consistent estimator of the parameter $s > 1$, and let $F(t;\widehat{s}_n)$
be the cumulative distribution function of the Zeta$(\widehat{s}_n)$
distribution.

We consider two asymptotically equivalent versions of the Cram\'er--von Mises
test statistic, namely
\[
    C_n^{a}
    = n \sum_{k=1}^{\ell} \bigl(\widehat{F}_n(k) - F(k;\widehat{s}_n) \bigr)^2
      \PP[X = k; \widehat{s}_n] + R(\ell),
\]
where $\PP[X = k; \widehat{s}_n]$, $k \geq 1$, is the Zeta$(\widehat{s}_n)$
probability mass function, and $R(\ell)$ is a remainder term, and
\[
    C_{n}^{e}
    = n \sum_{k=1}^{\infty} \bigl(\widehat{F}_n(k) - F(k;\widehat{s}_n) \bigr)^2
      \bigl(\widehat{F}_n(k) - \widehat{F}_n(k-1) \bigr),
\]
where
\[
 \widehat{F}_n(k) - \widehat{F}_n(k-1)
 = \frac{1}{n} \sum_{j=1}^{n} \boldsymbol{1}\{X_j = k\}, \quad k \geq 1,
\]
is the empirical probability mass function. The remainder term $R(\ell)$ of
$C_n^{a}$ depends on $\ell$ and on the true value of $s$. \citet{H:1996}
recommends choosing $\ell$ sufficiently large, in particular at least so large
that $n\bigl(1 - F(\ell;\widehat{s}_n)\bigr) \leq 10^{-4}$. Note that, for $C_n^{e}$, there are at most finitely many nonzero summands,
namely at most $M = \max\{X_1,\ldots,X_n\}$, and there is no need to choose any
upper limit $\ell$ for the sum, in contrast to $C_n^{a}$.

The test of \citet{Yang:2018} cannot be directly compared with our procedure, since it is
developed for a finite sample space and a simple null hypothesis with fully specified
probabilities. In order to obtain a method that is applicable to the Zeta distribution with
unknown parameter $s$, we adapt their kernelized Stein discrepancy by plugging in an
estimate $\widehat{s}_n$ of $s$ into the discrete Stein operator and working on the finite
support $\{1,\dots,K_n\}$, where $K_n = \max_{1\le i\le n} X_i$. Define the forward and
backward neighbours
\[
  \tau(x) =
  \begin{cases}
    x+1, & x < K_n,\\
    1,   & x = K_n,
  \end{cases}
  \qquad
  \rho(x) =
  \begin{cases}
    x-1, & x > 1,\\
    K_n, & x = 1,
  \end{cases}
\]
and the discrete score
\[
  s_{\widehat{s}_n}(x)
  = 1 - \Bigl(\frac{x}{\tau(x)}\Bigr)^{\widehat{s}_n}, \qquad x\in\{1,\dots,K_n\}.
\]
Let $k$ be a positive definite kernel on $\{1,\dots,K_n\}$. The corresponding Stein
kernel is
\[
\begin{aligned}
  \kappa_{\widehat{s}_n}(x,x')
  &= s_{\widehat{s}_n}(x)\,k(x,x')\,s_{\widehat{s}_n}(x')
   - s_{\widehat{s}_n}(x)\,\bigl(k(x,x') - k(x,\rho(x'))\bigr) \\
  &\quad
   - \bigl(k(x,x') - k(\rho(x),x')\bigr)\,s_{\widehat{s}_n}(x')
   + \bigl(k(x,x') - k(\rho(x),x') - k(x,\rho(x')) + k(\rho(x),\rho(x'))\bigr),
\end{aligned}
\]
and the Yang-type kernel Stein discrepancy with estimated parameter is given by the
degenerate U-statistic
\begin{equation}\label{eq:test_Yang18}
  S_{\mathrm{KSD}}
  = \frac{1}{n(n-1)} \sum_{j\neq i}^n
      \kappa_{\widehat{s}_n}(X_i,X_j).
\end{equation}
We implement a Gaussian kernel $k(x,x')=\exp(- (x-x')^2/2)$.

The goodness-of-fit test proposed by \citet{M:2009} is based on the Mellin
transform and was originally developed for both continuous and discrete Pareto-type
laws under a simple null hypothesis. In the discrete case, the null hypothesis is quivalent to \eqref{eq:H0} and the test statistic involves the difference between the empirical Mellin transform and
its theoretical counterpart. For a sample $X_1,\dots,X_n$ and a weight function
$w_\beta(t) = \exp(-\beta t)$ with $\beta>0$, the empirical Mellin transform is
\[
  M_n(t) = \frac{1}{n}\sum_{j=1}^n X_j^{-t}, \qquad t\ge 0.
\]

Adapting the discrete-Zeta version of \citet{M:2009} to the composite hypothesis
with unknown parameter $s$, we plug in an estimator $\widehat{s}_n$ of $s$ into the test
statistic. Using the integral representation (cf.\ formulas (1.3) and (3.1) in
\cite{M:2009}), the resulting Meintanis statistic is
\begin{equation}\label{eq:test_Meintanis}
  Z_{n,\beta}
  \;=\;
  n \int_0^\infty
    \Bigl[
       \zeta^2(\widehat{s}_n)\,M_n^2(t)
       + \zeta^2(\widehat{s}_n + t)
       - 2\,\zeta(\widehat{s}_n)\,\zeta(\widehat{s}_n + t)\,M_n(t)
    \Bigr]
    e^{-\beta t}\,dt,
\end{equation}
where $\zeta(\cdot)$ denotes the Riemann zeta function. In practice, the integral in
\eqref{eq:test_Meintanis} is evaluated numerically for each bootstrap sample and for each
chosen value of the tuning parameter $\beta$.

Finally, we adapt the characterization-based approach of \citet{BEN:2022}
to the Zeta family. For $s>2$ the Zeta pmf $p_s(k)=k^{-s}/\zeta(s)$ satisfies the
Stein-type identity
\[
  \mathbb{P}(X=k)
  =
  \E\!\left[
    \Bigl(1-\Bigl(\frac{X}{X+1}\Bigr)^{s}\Bigr)\mathbf{1}_{\{X\ge k\}}
  \right],\qquad k\in\mathbb{N},
\]
and deviations from this identity can be used to measure lack of fit. Given a sample
$X_1,\dots,X_n$ and an estimator $\widehat{s}_n$ of $s$, we define the empirical pmf
$\rho_n(k)=n^{-1}\sum_{j=1}^n \mathbf{1}_{\{X_j=k\}}$ and
\[
  e_n(k;\widehat{s}_n)
  =
  \frac{1}{n}\sum_{j=1}^n
    \Bigl(1-\Bigl(\frac{X_j}{X_j+1}\Bigr)^{\widehat{s}_n}\Bigr)
    \mathbf{1}_{\{X_j\ge k\}},
  \qquad k\in\mathbb{N}.
\]
The BEN test statistic is the squared $L^2$–distance
\[
  S_{\mathrm{BEN}}(\widehat{s}_n)
  =
  \sum_{k=1}^{M_n} \bigl(e_n(k;\widehat{s}_n)-\rho_n(k)\bigr)^2,
  \qquad M_n = \max_{1\le j\le n} X_j,
\]
which vanishes under the model and becomes large when the characterization is violated.
Critical values are obtained by a parametric bootstrap from the fitted Zeta distribution
$p_{\widehat{s}_n}$, exactly in analogy with the procedures proposed by \citet{BEN:2022}.

\begin{remark}
The general characterization in Theorem 3.1 of \citet{BEN:2022} requires a finite first moment
(cf.\ condition (C3)), so its assumptions are rigorously satisfied for the Zeta
family only in the parameter range $s>2$. For $1<s\le 2$ the Zeta law has infinite mean, so
condition (C3) is violated and the formal applicability of the theorem is no longer
guaranteed. Nevertheless, the above identity continues to hold at a heuristic level and
the corresponding statistic remains well defined, although a rigorous asymptotic theory in this
heavy-tailed regime would require additional arguments beyond those of
\citet{BEN:2022}.
\end{remark}

We consider several alternatives to explore the power of the test, all supported on
$\mathbb{N}$.  First, for $s>2$ we write
$\mathrm{Geom}(s)$ for the geometric distribution on $\{1,2,\dots\}$ with success
probability $p_s$ chosen such that its mean coincides with that of $\mathrm{Zeta}(s)$,
that is $1/p_s = \E[Z] = \zeta(s-1)/\zeta(s)$, $Z\sim\mathrm{Zeta}(s)$.  The notation Zipf$(s,N)$ refers to the
truncated Zipf distribution on $\{1,\dots,N\}$ with pmf
\[
  \mathbb{P}(X=k) = \frac{k^{-s}}{\sum_{j=1}^N j^{-s}}, \qquad k=1,\dots,N.
\]
For $s>1$, $N\in\mathbb{N}$ and
$p_{\mathrm{g}}\in(0,1)$, the Zeta/Geometric splice $\mathrm{ZG}(s,N,p_{\mathrm{g}})$
coincides with $\mathrm{Zeta}(s)$ on $\{1,\dots,N\}$ and has a shifted geometric tail
beyond $N$, i.e.\ for $k>N$ the probabilities are proportional to
$(1-p_{\mathrm{g}})^{k-N-1}p_{\mathrm{g}}$, with a global normalizing constant chosen
so that the pmf sums to one.  Conversely, for $s>1$ and $N\in\mathbb{N}$ the
Geometric/Zeta splice $\mathrm{GZ}(s,N)$ has a geometric probability mass on $\{1,\dots,N\}$ and a
shifted Zeta tail beyond $N$, so that for $k>N$ the probabilities are proportional to
$(k-N)^{-s}$ and the overall pmf is properly normalized.  Finally, the zigzag
alternatives $\mathrm{Zigzag}(s,\varepsilon)$, with $s>1$ and $|\varepsilon|<1$, are
defined by the pmf
\[
  \mathbb{P}(X=k) \propto k^{-s}\bigl(1+\varepsilon(-1)^k\bigr), \qquad k\in\mathbb{N},
\]
which corresponds to alternating up- and down-perturbations of the Zeta$(s)$ pmf.

The results of the simulation study are presented in Table \ref{tab:100}. The first result we can see is that no
test statistic uniformly dominates the others, which is not surprising, since there can't exist an omnibus testing procedure as is shown by \cite{J:2000}. Overall, all considered tests maintain the nominal significance level reasonably well, keeping the empirical size close to $0.05$. For benchmark alternatives such as the Poisson, binomial, discrete uniform, and negative binomial distributions, all procedures achieve rejection rates close to $100\%$, so these cases are not reported in Table~\ref{tab:100}. In general, most tests exhibit good power when deviations from the Zeta law occur in the left tail, where the probabilities are relatively large. Somewhat surprisingly, the test of \citet{Yang:2018} is the strongest procedure against the Geometric alternative when $s=2.5$, but its power drops essentially to zero when $s=3$ or $s=3.5$. A decline in power for the Geometric alternative is visible for all competitors as $s$ increases; in contrast, our statistics $T_{n,k}$ retain comparatively high power because, for larger $s$ (and thus larger $p_s$), the probabilities at $1$ and $2$ are quite substantial, and our test is particularly sensitive to deviations from the Zeta model at small values of $X$. All tests fail to detect the Geometric right-tail alternatives if the mixture split between the Zeta and Geometric components occurs at $k=10$. When the split is at $k=5$, the competing procedures perform visibly better, while our tests still achieve rejection rates only slightly above the nominal level; this is consistent with the fact that our test gives relatively small weight to discrepancies in $p(k)=\mathbb P(X=k)$ for large $k$, and suggests that alternative weightings might be explored in future work. On the other hand, all tests perform very well when the split is reversed, that is, when the Geometric component governs the left part of the distribution and the Zeta component the right tail. The truncated Zeta (Zipf) distribution is well detected if the truncation point is close to zero; however, all of our statistics show a relatively rapid loss of power as the ratio $N/s$ decreases, that is, when the truncation is so close to one that observations with substantial probability mass are no longer visible. Finally, the new statistic $T_{n,\beta}$ clearly outperforms the other procedures for the Zigzag alternatives, which again reflects the particular strength of our test in situations where the relative frequencies of small count values deviate from those implied by the Zeta distribution.

\begin{table}[t]
\setlength{\tabcolsep}{1.25mm}
\centering
\begin{tabular}{l|rrrrrrrrrrrrrrr}
 Alt. & $T_{n,0}$ & $T_{n,1}$ & $T_{n,2}$ & $T_{n,3}$ & $T_{n,4}$ & $T_{n,5}$ & $Z_{n,1}$ & $Z_{n,2}$ & $Z_{n,3}$ & $Z_{n,4}$ & $Z_{n,5}$ & $Z_{n,6}$ & $C_n^e$& $S_{\mathrm{KSD}}$ & $S_{\mathrm{BEN}}$\\
  \hline
  Zeta(1.5) & 5 & 5 & 5 & 5 & 5 & 5 & 5 & 4 & 4 & 4 & 5 & 5 & 5 & 5 & 5 \\
  Zeta(1.75) & 5 & 5 & 5 & 5 & 5 & 5 & 4 & 4 & 4 & 4 & 4 & 4 & 5 & 5 & 5 \\
  Zeta(2) & 5 & 5 & 5 & 4 & 4 & 5 & 4 & 4 & 4 & 4 & 4 & 4 & 5 & 5 & 4 \\
  Zeta(2.25) & 5 & 5 & 5 & 5 & 5 & 5 & 4 & 4 & 3 & 3 & 3 & 4 & 4 & 5 & 4 \\ \hline
  Geom(2.5) & 93 & 93 & 92 & 91 & 91 & 90 & 99 & 99 & 99 & 99 & 99 & 99 & 100 & 100 & 99 \\
  Geom(3) & 77 & 76 & 75 & 75 & 75 & 74 & 73 & 69 & 66 & 64 & 64 & 67 & 80 & 0 & 73 \\
  Geom(3.5)& 48 & 48 & 48 & 48 & 48 & 48 & 37 & 34 & 30 & 29 & 27 & 27 & 45 & 0 & 39 \\
ZG(3,10,0.2) & 4 & 4 & 5 & 4 & 4 & 5 & 3 & 3 & 3 & 3 & 3 & 3 & 4 & 6 & 4 \\
  ZG(3,10,0.4) & 4 & 5 & 5 & 5 & 5 & 5 & 3 & 3 & 2 & 2 & 2 & 3 & 4 & 7 & 4 \\
  ZG(3,10,0.6) & 5 & 5 & 5 & 5 & 5 & 5 & 3 & 3 & 2 & 2 & 2 & 3 & 4 & 7 & 5 \\
  ZG(3,10,0.8) & 5 & 5 & 5 & 5 & 5 & 5 & 3 & 3 & 3 & 2 & 3 & 3 & 4 & 8 & 4 \\
  ZG(2,5,0.2) & 5 & 5 & 5 & 5 & 5 & 5 & 12 & 14 & 18 & 22 & 26 & 32 & 9 & 42 & 7 \\
  ZG(2,5,0.4) & 6 & 5 & 6 & 5 & 5 & 5 & 25 & 29 & 33 & 41 & 47 & 51 & 18 & 90 & 24 \\
  ZG(2,5,0.6) & 6 & 6 & 6 & 6 & 6 & 6 & 32 & 36 & 41 & 47 & 54 & 59 & 27 & 99 & 49 \\
  ZG(2,5,0.8) & 6 & 6 & 6 & 6 & 6 & 5 & 38 & 42 & 47 & 53 & 60 & 66 & 35 & 100 & 74 \\
  ZG(2,30,0.4) & 5 & 5 & 5 & 5 & 5 & 5 & 4 & 4 & 5 & 7 & 9 & 12 & 4 & 17 & 5 \\
  GZ(2.5,5) & 93 & 92 & 91 & 91 & 90 & 90 & 99 & 99 & 99 & 99 & 99 & 100 & 100 & 96 & 99 \\
  GZ(2.5,10) & 93 & 92 & 91 & 91 & 90 & 90 & 99 & 99 & 99 & 99 & 99 & 100 & 100 & 100 & 99 \\
  GZ(2.5,20) & 93 & 92 & 92 & 91 & 91 & 90 & 99 & 99 & 99 & 99 & 99 & 99 & 99 & 100 & 98 \\
  GZ(2.5,50) & 94 & 93 & 92 & 92 & 91 & 91 & 99 & 99 & 99 & 99 & 99 & 99 & 100 & 100 & 99 \\
  Zipf(1.5,10) & 28 & 27 & 26 & 25 & 24 & 24 & 84 & 87 & 90 & 93 & 95 & 97 & 84 & 100 & 75 \\
  Zipf(1.75,10) & 18 & 17 & 17 & 16 & 16 & 16 & 58 & 61 & 65 & 70 & 75 & 79 & 54 & 100 & 42 \\
  Zipf(2,10) & 13 & 13 & 13 & 13 & 13 & 13 & 33 & 35 & 36 & 41 & 46 & 52 & 28 & 64 & 22 \\
  Zipf(2.25,10) & 8 & 8 & 8 & 8 & 8 & 8 & 16 & 16 & 16 & 18 & 23 & 26 & 14 & 20 & 12 \\
  Zipf(2.5,10) & 7 & 7 & 7 & 7 & 7 & 7 & 8 & 8 & 8 & 8 & 10 & 11 & 8 & 11 & 7 \\
  Zipf(3,10) & 5 & 5 & 5 & 5 & 5 & 5 & 4 & 3 & 3 & 3 & 3 & 3 & 4 & 8 & 5 \\
  Zipf(1.5,20) & 15 & 14 & 14 & 13 & 13 & 13 & 70 & 78 & 85 & 90 & 95 & 98 & 67 & 100 & 45 \\
  Zipf(1.75,20) & 10 & 10 & 10 & 9 & 9 & 9 & 36 & 42 & 48 & 55 & 61 & 69 & 30 & 99 & 20 \\
  Zipf(2,5) & 33 & 31 & 31 & 30 & 30 & 29 & 61 & 61 & 61 & 62 & 66 & 69 & 59 & 28 & 53 \\
  Zipf(2.25,5) & 21 & 20 & 20 & 20 & 20 & 20 & 40 & 40 & 39 & 39 & 41 & 45 & 39 & 8 & 33 \\
  Zipf(2.5,5) & 14 & 14 & 14 & 14 & 13 & 13 & 23 & 22 & 21 & 21 & 22 & 25 & 22 & 6 & 18 \\
  Zipf(3,5) & 8 & 8 & 8 & 7 & 7 & 7 & 7 & 7 & 7 & 7 & 7 & 7 & 8 & 5 & 7 \\
  Zigzag(1.75,0.1) & 12 & 12 & 12 & 12 & 12 & 12 & 8 & 8 & 10 & 11 & 9 & 5 & 7 & 5 & 11 \\
  Zigzag(1.75,0.5) & 96 & 96 & 96 & 96 & 96 & 96 & 89 & 85 & 75 & 56 & 27 & 12 & 93 & 20 & $\ast$\\
  Zigzag(2,0.1) & 11 & 12 & 12 & 12 & 12 & 12 & 8 & 8 & 9 & 10 & 11 & 12 & 9 & 7 & 11 \\
  Zigzag(2,0.5) & 96 & 96 & 97 & 97 & 97 & 97 & 87 & 84 & 83 & 78 & 66 & 45 & 93 & 25 & 98 \\
  Zigzag(2.5,0.1) & 10 & 10 & 10 & 10 & 10 & 11 & 7 & 7 & 6 & 6 & 6 & 8 & 10 & 6 & 9 \\
  Zigzag(2.5,0.5) & 95 & 95 & 95 & 95 & 95 & 95 & 81 & 76 & 74 & 75 & 76 & 74 & 91 & 26 & 95 \\
  Zigzag(3,0.1) & 9 & 9 & 9 & 9 & 9 & 9 & 6 & 5 & 5 & 5 & 5 & 5 & 8 & 5 & 7 \\
  Zigzag(3,0.5) & 90 & 90 & 90 & 90 & 90 & 90 & 70 & 64 & 59 & 59 & 61 & 63 & 84 & 5 & 88 \\
\end{tabular}
\caption{Empirical rejection rates for the different test statistics for a sample size of $n=100$ and a significance level of $5\%$. Every entry is based on 10000 replications. The entry with $\ast$ could not be computed due to numerical instability.}\label{tab:100}
\end{table}

\section{Comments and Outlook}\label{sec:out}
We have proposed a new family of test statistics for assessing the composite hypothesis that an integer–valued data set on $\N$ arises from an unspecified member of the zeta family. We derived the asymptotic distribution of the corresponding test and established the consistency of a parametric bootstrap approximation. Moreover, we conducted what appears to be the first comprehensive Monte Carlo study for this testing problem: although \cite{M:2009} outlines a bootstrap procedure, simulation results are only reported for the simple null hypothesis. Our experiments demonstrate that the proposed tests are strong competitors to existing methods. In addition, we extended the kernelized Stein discrepancy approach of \cite{Yang:2018} to the composite setting and introduced, in the spirit of \cite{BEN:2022}, a previously unavailable characterization-based test tailored to the zeta distribution.

We finalize the article by pointing out some open research directions. In Remark \ref{rem:KLE} we pointed out the connection of the limit null distribution to a weighted sum of $\chi^2_1$ distributed random variables. To use the Rayleigh-Ritz method one needs an orthonormal basis of $L^2_w$, which depends on the underlying weight function. In the considered case $w(t)=(1-t)^\beta$, $\beta\ge0$, a starting point would be the normalized shifted Jacobi polynomials
\begin{equation*}
    \varphi_{n,\beta}(x)=\sqrt{2n+\beta+1}\sum_{k=0}^n (-1)^k \binom{n+\beta}{n-k}\binom{n}{k}(1-t)^k\, t^{\,n-k}, \quad t \in[0,1].
\end{equation*}
could lead to a precise approximation of the eigenvalues of the covariance operator, see \cite{EJM:2025} for the methodology and examples of a variety of weighted $L^2$ spaces, covariance kernels of Gaussian processes, and other families of orthogonal polynomials.

In the simulation study we exclusively employed the maximum likelihood estimator for the shape parameter of the Zeta distribution. Since \citet{DKO:1990} indicate that the choice of estimation procedure can substantially affect the power of goodness-of-fit tests, it would be of interest to explore alternative estimators, such as minimum distance estimators or generalized method-of-moments type procedures inspired by Stein’s method as in \citet{BEK:2021,EFGPS:2025}.

Obviously, the Zeta distribution is the infinite version of the truncated Zipf distribution, so it would be interesting to consider the testing problem of fit to the truncated Zipf family, either for fixed truncation parameter $N$ or unknown $N$. The corresponding rates of the birth- and death process in anology to the derivation in the introduction would then stay the same but the index $k$ is restricted to $\{1,\ldots,N\}$. Another interesting generalization of the Zeta distribution is the so called Lerch distribution, which can also be expressed as stationary limit distribution of a birth- and death process, for details see \cite{KPH:2010}.

\section*{Acknowledgement}
The work of DH was supported by the Czech Science Foundation project GAČR No. 25-15844S.

\bibliographystyle{apalike}

\bibliography{lit_Zipf}

\end{document}